\documentclass[letterpaper, 10 pt, conference]{ieeeconf}

\def\bsa{{\boldsymbol{a}}}
\def\bsb{{\boldsymbol{b}}}
\def\bsc{{\boldsymbol{c}}}

\def\bsu{{\boldsymbol{u}}}
\def\bsv{{\boldsymbol{v}}}

\def\bsx{{\boldsymbol{x}}}

\def\bsz{{\boldsymbol{z}}}

\def\bsF{{\boldsymbol{F}}}

\def\bsI{{\boldsymbol{I}}}

\def\bsL{{\boldsymbol{L}}}

\def\bsX{{\boldsymbol{X}}}

\def\calA{{\mathcal{A}}}

\def\calG{{\mathcal{G}}}

\def\calP{{\mathcal{P}}}

\usepackage[cmex10]{amsmath}
\usepackage{amssymb}
\usepackage{algorithm,algorithmic}

\usepackage{silence}
\usepackage{tikz}
\usepackage{tkz-berge}
\usepackage{lipsum}
\usetikzlibrary{calc}
\usetikzlibrary{shapes,fit}
\usepackage{verbatim}
\usepackage{setspace}
\usetikzlibrary{shapes,arrows}
\usepackage[short]{optidef}

\usepackage{amsthm}

\usepackage{color}
\definecolor{gray}{RGB}{128,128,128}

\usepackage{cite}
\usepackage{graphicx}
\graphicspath{{figure/}}
\usepackage{epstopdf}
\usepackage{tikz}
\usetikzlibrary{arrows,shapes,backgrounds}

\usepackage{bm}

\usepackage[caption=false,subrefformat=parens,labelformat=parens]{subfig}
\newtheorem{theorem}{Theorem}
\newtheorem{assumption}{Assumption}

\newtheorem{lemma}{Lemma}

\newtheorem{definition}{Definition}
\newtheorem{remark}{Remark}
\newtheorem{example}{Example}

\DeclareMathOperator{\nullrank}{null}
\DeclareMathOperator{\rank}{rank}

\DeclareMathOperator{\col}{col}

\DeclareMathOperator{\diag}{diag}
\DeclareMathOperator{\Deg}{Deg}

\DeclareMathOperator{\image}{image}

\DeclareMathOperator*{\argmin}{arg\, min}

\def\R{\mathbb{R}}

\newenvironment{proof}[1][Proof]%
  {\smallskip\par\noindent\textbf{#1\,:\ }}%
  {\hspace*{\fill} \rule{6pt}{6pt}\smallskip}
\newenvironment{proof*}[1][Proof]%
  {\smallskip\par\noindent\textbf{#1\,:\ }}%

\newlength{\fwidth}
\begin{document}
\IEEEoverridecommandlockouts
\title{Exponential Convergence for Distributed Smooth Optimization\\ Under the Restricted Secant Inequality Condition}
\author{Xinlei Yi, Shengjun Zhang, Tao Yang, Karl H. Johansson, and Tianyou Chai
\thanks{This work was supported by the
Knut and Alice Wallenberg Foundation, the  Swedish Foundation for Strategic Research, the Swedish Research Council.}
\thanks{X. Yi and K. H. Johansson are with the Division of Decision and Control Systems, School of Electrical Engineering and Computer Science, KTH Royal Institute of Technology, 100 44, Stockholm, Sweden. {\tt\small \{xinleiy, kallej\}@kth.se}.}%
\thanks{S. Zhang is with the Department of Electrical Engineering, University of North Texas, Denton, TX 76203 USA. {\tt\small  ShengjunZhang@my.unt.edu}.}
\thanks{T. Yang and T. Chai are with the State Key Laboratory of Synthetical Automation for Process Industries, Northeastern University, 110819, Shenyang, China. {\tt\small \{yangtao,tychai\}@mail.neu.edu.cn}.}
}

\maketitle

\begin{abstract}
This paper considers the distributed smooth optimization problem in which the objective is to minimize a global cost function formed by a sum of local smooth cost functions,  by using local information exchange. The standard assumption for proving exponential/linear convergence of first-order methods is the strong convexity of the cost functions, which does not hold for many practical applications.  In this paper, we first show that the continuous-time distributed  primal-dual gradient algorithm converges to one global minimizer exponentially  under the assumption that the global cost function satisfies the restricted secant inequality condition. This condition is weaker than the strong convexity condition since it does not require convexity and the global minimizers are not necessary to be unique. We then show that the discrete-time distributed primal-dual algorithm constructed by using the Euler's approximation method converges to one global minimizer linearly under the same condition. The theoretical results are illustrated by numerical simulations.
\end{abstract}

\section{Introduction}
The distributed optimization problem has a long history which can be traced back to \cite{tsitsiklis, Tsitsiklis_TAC86,bertsekas1989parallel}.
Such a problem has gained renewed interests in recent years due to its wide applications on power system, machine learning, and sensor network, just to name a few \cite{sayed2014adaptation,nedich2015convergence}.

When the cost functions are convex, various distributed optimization algorithms have been developed for solving this problem and can be divided into two categories depending on whether the algorithm is discrete-time or continuous-time.
Most existing distributed optimization algorithms are discrete-time 
and are based on the consensus and distributed (sub)gradient descent method \cite{Kalle-CDC08,Nedic09,Minghui-TAC12,Rabbat-PushSum-CDC2012,Nedic15, Yang-EDP-Delay}.
Although the distributed (sub)gradient descent algorithms can deal with non-smooth convex functions and has been extended in several directions to handle more realistic scenarios, the convergence rate is at most sub-linear due to the diminishing stepsizes.
With a fixed stepsize, the distributed (sub)gradient descent algorithms converge fast, but only to a neighborhood of an optimal point \cite{Baras-constant-step-size,Yin-DGD}.
Recent studies focused on developing accelerated algorithms with fixed stepsizes by using some sort of historical information \cite{jakovetic2015linear,nedic2017achieving,nedic2017geometrically,qu2018harnessing,qu2017accelerated,xi2018add,xu2018convergence,xin2018linear,pu2018push,jakovetic2019unification,Damiano-TAC2016,saadatniaki2018optimization,shi2015extra,zeng2017extrapush,xi2017dextra,Yang-PESGM2018,LishaYao-ICCA2018,xin2019frost}.

Although most existing distributed optimization algorithms are discrete-time, with the development of cyber-physical systems, continuous-time algorithms have also been proposed, mainly because many practical systems such as robots and unmanned vehicles operate in continuous-time and the well-developed continuous-time control techniques (in particular Lyapunov stability theory) may facilitate the analysis.
The existing continuous-time distributed algorithms can be classified into two classes depending on whether the algorithm uses the first-order gradient information \cite{Elia_Allerton10, Cortes_TAC_CT,yu2016gradient, kia2015distributed,zhang2017distributed,li2018distributed,Yi2018distributed,liang2019exponential} or the second-order Hessian information \cite{Jie-ZGS,ErminWei-TAC1}.

Among these distributed optimization algorithms, the standard assumption for proving exponential/linear convergence are that each local cost function is smooth and (local or global) cost functions are strongly convex. For example, in \cite{xin2019frost,Jie-ZGS,yu2016gradient,kia2015distributed,zhang2017distributed,jakovetic2015linear,nedic2017achieving,nedic2017geometrically,qu2018harnessing,qu2017accelerated,xi2018add,xu2018convergence,xin2018linear,pu2018push,jakovetic2019unification}, the authors assumed that each local cost function is strongly convex and in \cite{Damiano-TAC2016,li2018distributed,saadatniaki2018optimization}, the authors assumed that the global cost function is strongly convex. Unfortunately, many practical applications, such as least squares and logistic regression, do not always have strongly convex cost functions \cite{yang2018distributed}. This situation has motivated researchers to consider alternatives to strong convexity. There are some results in centralized optimization. For instance, in \cite{necoara2019linear}, the authors derived linear convergence rates of several centralized first-order methods for solving the smooth convex constrained optimization problem under the quadratic function growth condition and in \cite{karimi2016linear}, the authors showed linear convergence rates of centralized proximal-gradient methods for solving the smooth (non-convex) optimization problem under the assumption that the cost function satisfies the Polyak-{\L}ojasiewicz condition. However, to the best of knowledge, there are few such kind of results in distributed optimization except \cite{shi2015extra,liang2019exponential}. In \cite{shi2015extra}, the authors proposed the distributed exact first-order algorithm (EXTRA) to solve smooth convex optimization and proved linear convergence rates under the condition that the global cost function is restricted strongly convex and the optimal set is a singleton. In \cite{liang2019exponential}, the authors established exponential/linear convergence of the distributed primal-dual gradient decent algorithm for solving smooth convex optimization under the condition that the primal-dual gradient map is metrically subregular which is weaker than strict or strong convexity.

In this paper, we consider the problem of solving distributed smooth optimization and analyse the convergence rate of the distributed primal-dual gradient decent algorithm. We first show that the continuous-time distributed  primal-dual gradient algorithm converges to one global minimizer exponentially  under the assumption that the global cost function satisfies the restricted secant inequality condition. This condition is weaker than the (restrict) strong convexity condition assumed in \cite{jakovetic2015linear,nedic2017achieving,nedic2017geometrically,qu2018harnessing,qu2017accelerated,xi2018add,xu2018convergence,xin2018linear,pu2018push,jakovetic2019unification,Damiano-TAC2016,saadatniaki2018optimization,zeng2017extrapush,xi2017dextra,shi2015extra,Jie-ZGS,yu2016gradient,kia2015distributed,zhang2017distributed,li2018distributed,Yi2018distributed} since it does not require convexity and the global minimizers are not necessarily to be unique and it is different from metric subregularity criterion assumed in \cite{liang2019exponential}. We then show that the discrete-time counterpart of the continuous-time distributed  primal-dual gradient algorithm, which is obtained from a simple discretization by Euler's method, also converges to one global minimizer linearly under the same condition.

The rest of this paper is organized as follows.
Section~\ref{sec-preliminary} introduces some preliminaries.
Section~\ref{sec-problem} gives problem formulation and assumptions.
The main results are stated in Sections~\ref{sec-main} and \ref{sec-main-dc}.
Simulations are given in Section~\ref{sec-simulation}.
Finally, concluding remarks are offered in Section~\ref{sec-conclusion}.

\noindent {\bf Notations}: $[n]$ denotes the set $\{1,\dots,n\}$ for any positive constant $n$. $\col(z_1,\dots,z_k)$ is the concatenated column vector of vectors $z_i\in\mathbb{R}^{p_i},~i\in[k]$. ${\bf 1}_n$ (${\bf 0}_n$) denotes the column one (zero) vector of dimension $n$. $\bsI_n$ is the $n$-dimensional identity matrix. Given a vector $[x_1,\dots,x_n]^\top\in\mathbb{R}^n$, $\diag([x_1,\dots,x_n])$ is a diagonal matrix with the $i$-th diagonal element being $x_i$. The notation $A\otimes B$ denotes the Kronecker product
of matrices $A$ and $B$. $\rank(A)$, $\image(A)$, and $\nullrank(A)$ are the rank, image, and null of matrix $A$, respectively.
Given two symmetric matrices $M,N$, $M\ge N$ means that $M-N$ is positive semi-definite. $\rho(\cdot)$ stands for the spectral radius for matrices and $\rho_2(\cdot)$ indicates the minimum
positive eigenvalue for matrices having positive eigenvalues. $\|\cdot\|$ represents the Euclidean norm for
vectors or the induced 2-norm for matrices. For given positive semi-definite matrix $A$, $\|x\|_A$ denotes the norm $\sqrt{x^\top Ax}$.  Given a differentiable function $g$, $\nabla g$ denotes the gradient of $g$. 

\section{Preliminaries}\label{sec-preliminary}
In this section, we present some definitions from algebraic graph theory \cite{mesbahi2010graph},  the restricted secant inequality \cite{zhang2015restricted}, and monotonicity properties of vector functions \cite{crouzeix2000conditions}.

\subsection{Algebraic Graph Theory}

Let $\mathcal G=(\mathcal V,\mathcal E, A)$ denote a weighted undirected graph with the set of vertices (nodes) $\mathcal V =[n]$, the set of links (edges) $\mathcal E
\subseteq \mathcal V \times \mathcal V$, and the weighted adjacency matrix
$A =A^{\top}=(a_{ij})$ with nonnegative elements $a_{ij}$. A link of $\mathcal G$ is denoted by $(i,j)\in \mathcal E$ if $a_{ij}>0$, i.e., if vertices $i$ and $j$ can communicate with each other. It is assumed that $a_{ii}=0$ for all $i\in [n]$. Let $\mathcal{N}_i=\{j\in [n]:~ a_{ij}>0\}$ and $\deg_i=\sum\limits_{j=1}^{n}a_{ij}$ denotes the neighbor set and weighted degree of vertex $i$, respectively. The degree matrix of graph $\mathcal G$ is $\Deg=\diag([\deg_1, \cdots, \deg_n])$. The Laplacian matrix is $L=(L_{ij})=\Deg-A$. A  path of length $k$ between vertices $i$ and $j$ is a subgraph with distinct vertices $i_0=i,\dots,i_k=j\in [n]$ and edges $(i_j,i_{j+1})\in\mathcal E,~j=0,\dots,k-1$.
An undirected graph is  connected if there exists at least one path between any two vertices.

\subsection{Restricted Secant Inequality}
\begin{definition} (Definitions 1 and 2 in \cite{zhang2015restricted})
A differentiable function $f(x):~\mathbb{R}^p\mapsto\mathbb{R}$ satisfies the restricted secant inequality condition with constant $\nu>0$ if
\begin{align}
&(\nabla f(x)-\nabla f(\calP_{X^*}(x))^\top(x-\calP_{X^*}(x))\nonumber\\
&\ge \nu\|x-\calP_{X^*}(x)\|^2,~\forall x\in \mathbb{R}^p,\label{cdc18so:equ:rsc}
\end{align}
where $X^*$ is the set of all  global minimizers of $f$ and $\calP_{X^*}(x)$ is the projection of $x$ onto the set $X^*$, i.e., $\calP_{X^*}(x)=\argmin_{y\in X^*}\|x-y\|^2$. If the function $f$ is also convex it is called restricted strong convexity.
\end{definition}
Note that, unlike the strong convexity, the restricted secant inequality \eqref{cdc18so:equ:rsc} alone does not even imply the convexity of $f$. Moreover, it does not imply that $X^*$ is a singleton either. However, it implies that every stationary point is a global minimizer, i.e., $X^*=\{x\in\mathbb{R}^p:\nabla f(x)={\bf 0}_p\}$. Therefore, it is weaker than (essential and weak) strong convexity \cite{karimi2016linear}. Example~\label{nonconvex:example} in the following gives a function which satisfies the restricted secant inequality condition but is not convex.
See \cite{zhang2015restricted,necoara2019linear} for more examples of functions that satisfy the restricted secant inequality condition.

\begin{example}\label{nonconvex:example}
(Example~2 in \cite{zhang2015restricted})
\begin{align*}
f(x)=\begin{cases}
         0, & x\le0, \\
         1-\sqrt{1-x^2}, & 0\le x<\frac{\sqrt{2}}{2}, \\
         g_{1}(x), & \frac{\sqrt{2}}{2}\le x<1,\\
         g_{2}(x), & x\ge1,
       \end{cases}
\end{align*}
where $g_{1}(x)=\sqrt{1-(x-\sqrt{2})^2}-\sqrt{2}+1$, $g_{2}(x)=\frac{1}{2}(x-1+\sqrt{\frac{\sqrt{2}-1}{2}})^2+\sqrt{2\sqrt{2}-2}+\frac{5-5\sqrt{2}}{4}$.
\end{example}

\subsection{Monotonicity}
\begin{definition} (See Section 2.2 in \cite{crouzeix2000conditions})
A mapping $F:\mathbb{K}\subseteq\mathbb{R}^p\rightarrow\mathbb{R}^p$ is said to be
\begin{enumerate}
  \item pseudomonotone on $\mathbb{K}$ if for all $a,b\in\mathbb{K}$,
  \begin{align*}
    (a-b)^\top F(b)\ge0\Rightarrow(a-b)^\top F(a)\ge0;
  \end{align*}
  \item pseudomonotone$^+_*$ on $\mathbb{K}$ if it is pseudomonotone on $\mathbb{K}$ and for all $a,b\in\mathbb{K}$,
  \begin{align*}
    &[(a-b)^\top F(b)=0~\text{and}~(a-b)^\top F(a)=0]\\
    &\Rightarrow F(a)=F(b).
  \end{align*}
\end{enumerate}
\end{definition}
The gradient of a differentiable pseudoconvex function is  pseudomonotone \cite{karamardian1976complementarity,penot1997generalized} and the gradient of a differentiable G-convex function is pseudomonotone$^+_*$ \cite{crouzeix2000conditions}.

\section{Problem Formulation and assumptions}\label{sec-problem}
Consider a network of $n$ agents, each of which has a local cost function $f_i: \mathbb{R}^{p}\rightarrow \mathbb{R}$.
All agents collaborate together to find an optimizer $x^*$ that minimizes the global objective $f(x)=\sum_{i=1}^nf_i(x)$, i.e.,
\begin{equation}\label{eqn:xopt}
 \min_{x\in {\R}^p} f(x).
\end{equation}
The communication among agents is described by an undirected weighted graph $\mathcal{G}$. Throughout this paper, we assume that the undirected graph $\mathcal G$ is connected.
With a slight abuse of notation, let $X^*=\{x^*\}$ denote the optimal set of the optimization problem \eqref{eqn:xopt}. For simplicity, let $\bsx=\col(x_1,\dots,x_n)$, $\tilde{f}(\bm{x})=\sum_{i=1}^{n}f_i(x_i)$, $\bsX^*=\{{\bf 1}_n \otimes x^*:~x^*\in X^*\}$, and $\bsL=L\otimes \bsI_p$. The following assumptions are made.

\begin{assumption}\label{cdc18so:ass:filc} 
Each local cost function is differentiable. Moreover, the optimal set $X^*$ is nonempty and convex.
\end{assumption}

\begin{assumption}\label{cdc18so:ass:fiu}
Each  local cost function is smooth, that is, for each $i\in [n]$, $f_i$ has globally Lipschitz continuous gradient with constant $L_{f_i}>0$:
\begin{align*}
\|\nabla f_i(a)-\nabla f_i(b)\|\le L_{f_i}\|a-b\|,~\forall a,b\in \mathbb{R}^p.
\end{align*}
\end{assumption}

\begin{assumption}\label{cdc18so:ass:fil} The global cost function $f(x)$ satisfies the restricted secant inequality condition with constant $\nu>0$.
\end{assumption}

\begin{assumption}\label{cdc18so:ass:singleton} $\{\nabla\tilde{f}(\bsx):~\bsx\in\bsX^*\}$ is a singleton.
\end{assumption}

\begin{remark}
Compared with \cite{xin2019frost,Jie-ZGS,yu2016gradient,kia2015distributed,zhang2017distributed,li2018distributed,Yi2018distributed,liang2019exponential,jakovetic2015linear,nedic2017achieving,nedic2017geometrically,qu2018harnessing,qu2017accelerated,xi2018add,xu2018convergence,xin2018linear,pu2018push,jakovetic2019unification,Damiano-TAC2016,saadatniaki2018optimization,zeng2017extrapush,xi2017dextra,shi2015extra},
Assumptions~\ref{cdc18so:ass:filc}--\ref{cdc18so:ass:fiu} are mild since the convexity of the cost functions and the boundedness of their gradients are not assumed.
Assumption~\ref{cdc18so:ass:fil} only requires the global cost function rather than each local cost function  satisfies the restricted secant inequality condition. This is weaker than the assumptions used in \cite{xin2019frost,Jie-ZGS,yu2016gradient,kia2015distributed,zhang2017distributed,jakovetic2015linear,nedic2017achieving,nedic2017geometrically,qu2018harnessing,qu2017accelerated,xi2018add,xu2018convergence,xin2018linear,pu2018push,jakovetic2019unification} which assumed that each local cost function is strongly convex, and \cite{Damiano-TAC2016,li2018distributed,saadatniaki2018optimization} which assumed that the global cost function is strongly convex, and than \cite{zeng2017extrapush,xi2017dextra} which assumed that each local cost function is restricted strongly convex and the optimal set $X^*$ is a singleton, and \cite{shi2015extra,Yi2018distributed} which assumed that the global cost function is restricted strongly convex and $X^*$ is a singleton. One sufficient condition to guarantee that Assumption~\ref{cdc18so:ass:singleton} holds is that $X^*$ is a singleton. The following lemma gives another sufficient condition. Both sufficient conditions do not require the cost functions to be convex.
\end{remark}

From Proposition 14 in \cite{crouzeix2000conditions}, we have the following lemma.
\begin{lemma}\label{cdc18so:lemma:equigra}
Let $\mathbb{H}=\{{\bf 1}_n\otimes x:~x\in\mathbb{R}^{p}\}$.
Suppose that each local cost function is differentiable and $X^*$ is nonempty.  If $\nabla \tilde{f}$ is pseudomonotone$^+_*$ on $\mathbb{H}$, then $\{\nabla\tilde{f}(\bsx):~\bsx\in\bsX^*\}$ is a singleton.
\end{lemma}

\section{Continuous-time Distributed algorithm}\label{sec-main}
Noting that the Laplacian matrix $L$ is positive semi-definite and $\nullrank(L)=\{{\bf 1}_n\}$ since $\calG$ is connected, we know that
the optimization problem \eqref{eqn:xopt} is equivalent to the following constrained problem
\begin{mini}
{\bsx\in \mathbb{R}^{np}}{\tilde{f}(\bsx)}{\label{eqn:xoptcon}}{}
\addConstraint{\bsL^{1/2}\bsx=}{{\bf 0}_{np}.}{}
\end{mini}
Here, we use $\bsL^{1/2}\bsx={\bf 0}_{np}$ rather than $\bsL\bsx={\bf 0}_{np}$ as the constraint since it is also equivalent to $\bsx={\bf 1}_n\otimes x$ and it has a good property which will be shown in Remark~\ref{nonconvex:remark3}.

Let $\bsu=\col(u_1,\dots,u_n)\in\mathbb{R}^{np}$ denote the dual variable, then the augmented Lagrangian function associated with \eqref{eqn:xoptcon} is
\begin{align}\label{nonconvex:lagran}
\calA(\bsx,\bsu)=\tilde{f}(\bsx)+\frac{\alpha}{2}\bsx^\top\bsL\bsx+\beta\bsu^\top\bsL^{1/2}\bsx,
\end{align}
where $\alpha>0$ and $\beta>0$ are constants. Although $\tilde{f}(\bsx)$ does not satisfy the restricted secant inequality condition, the following lemma shows that $\tilde{f}(\bsx)+\frac{\alpha}{2}\bsx^\top\bsL\bsx$ satisfies the restricted secant inequality condition with respect with $\bsX^*$.
\begin{lemma}\label{cdc18so:prop}
Suppose that Assumptions~\ref{cdc18so:ass:filc}--\ref{cdc18so:ass:fil}. If $\alpha>\frac{2nL_{f}^2+\nu L_{f}}{\nu\rho_2(L)}$, where $L_{f}=\max_{i\in[n]}\{L_{f_i}\}$, then
   \begin{align}
     &(\nabla \tilde{f}(\bsx)-\nabla \tilde{f}(\calP_{\bsX^*}(\bsx)))^\top(\bsx-\calP_{\bsX^*}(\bsx))+\alpha\|\bsx\|^2_\bsL\nonumber\\
     &\ge\nu_1\|\bsx-\calP_{\bsX^*}(\bsx)\|^2,~\forall \bsx\in \mathbb{R}^{np},\label{rsc-eq}
   \end{align}
where $\nu_1=\min\{\frac{\nu}{2n},\alpha\rho_2(L)-\frac{2nL_{f}^2+\nu L_{f}}{\nu}\}>0$.
\end{lemma}
\begin{proof}
See Appendix~\ref{cdc18so:propproof}.
\end{proof}

\begin{remark}\label{cdc18so:remarkalpha}
Lemma~\ref{cdc18so:prop}  extends Proposition~3.6 in \cite{shi2015extra} and plays an important role in the proof of the exponential convergence later.
The key difference between Lemma~\ref{cdc18so:prop} and Proposition~3.6 in \cite{shi2015extra} is that here we do not assume that $\tilde{f}$ is convex  and $X^*$ is a singleton. The requirement that $\alpha>\frac{2nL_{f}^2+\nu L_{f}}{\nu\rho_2(L)}$ is used to eliminate the effects of non-convexity of  $\tilde{f}$. Similar to the proof of Proposition~3.6 in \cite{shi2015extra}, we can show that if $\tilde{f}$ is convex, then this requirement can be relaxed by $\alpha>0$ and \eqref{rsc-eq} still holds with
$\nu_1=\min\{\frac{\nu}{n}-2L_{f}\iota,\frac{\alpha\rho_2(L)\iota^2}{1+\iota^2}\}>0$, where $\iota\in(0,\frac{\nu}{2nL_{f}})$. Due to the similarity, we omit the details here.
\end{remark}

Based on the primal-dual gradient method, a continuous-time distributed algorithm to solve \eqref{eqn:xoptcon} is proposed as follows:
\begin{subequations}\label{kiau-algo-compact}
\begin{align}
\dot{\bm{x}}(t)&=-\alpha\bsL\bm{x}(t)-\beta\bsL^{1/2}\bm{u}(t)-\nabla \tilde{f}(\bm{x}(t)),\\
\dot{\bm{u}}(t)&=\beta\bsL^{1/2}\bm{x}(t),~\forall \bsx(0)\in\mathbb{R}^{np},~\bsu(0)={\bf 0}_{np}.
\end{align}
\end{subequations}
Denote $\bsv=\col(v_1,\dots,v_n)=\bsL^{1/2}\bm{u}$, then the algorithm \eqref{kiau-algo-compact} can be rewritten as
\begin{subequations}\label{kia-algo-compact}
\begin{align}
\dot{\bm{x}}(t)&=-\alpha\bsL\bm{x}(t)-\beta\bm{v}(t)-\nabla \tilde{f}(\bm{x}(t)),\\
\dot{\bm{v}}(t)&=\beta\bsL\bm{x}(t),~\forall \bsx(0)\in\mathbb{R}^{np},~\bsv(0)={\bf 0}_{np},
\end{align}
\end{subequations}
or
\begin{subequations}\label{kia-algo}
\begin{align}
\dot{x}_i(t) =& -\alpha\sum_{j=1}^nL_{ij}x_j(t)-\beta v_i(t)-\nabla f_i(x_i(t)),\\
\dot{v}_i(t) =& \beta\sum_{j=1}^n L_{ij}x_j(t),~\forall x_i(0)\in\mathbb{R}^p, \ v_i(0)={\bf 0}_p.
\end{align}
\end{subequations}

We have the following result for the continuous-time distributed primal-dual gradient decent algorithm \eqref{kia-algo}.
\begin{theorem}\label{thm2}
Each agent $i\in [n]$ runs the distributed algorithm \eqref{kia-algo}. If Assumptions~\ref{cdc18so:ass:filc}--\ref{cdc18so:ass:singleton} hold, $\alpha>\frac{2nL_{f}^2+\nu L_{f}}{\nu\rho_2(L)}$, and $\beta>0$, then $\bsx(t)$ exponentially converges to $\bsX^*$  with a rate no less than $\frac{\epsilon_2}{2\epsilon_3}>0$, where $\epsilon_2=\min\{\frac{\beta}{2},\epsilon_1\nu_1\}>0$ and $\epsilon_3=\max\{\frac{\epsilon_1}{\rho_2(L)}+\frac{\alpha}{2\beta}+\frac{1}{2},\epsilon_1+\frac{1}{2}\}$, with $\epsilon_1=\max\{\frac{1}{\nu_1}(\frac{L_{f}^2}{2\beta}+\rho(L)\beta),\frac{\beta}{\alpha}\}$.
\end{theorem}
\begin{proof}
The proof is given in Appendix~\ref{proof-thm2}.
\end{proof}

\begin{remark}\label{nonconvex:remark3}
If we use  $\bsL\bsx={\bf 0}_{np}$ as the constraint in \eqref{eqn:xoptcon}, then
we could construct an alternative continuous-time distributed primal-dual algorithm
\begin{subequations}\label{elia-algo}
\begin{align}
\dot{x}_i(t) =& -\alpha\sum_{j=1}^nL_{ij}x_j(t)-\beta\sum_{j=1}^nL_{ij}v_j(t)-\nabla f_i(x_i(t)), \label{elia-algo-x}\\
\dot{v}_i(t) =& \beta\sum_{j=1}^n L_{ij}x_j(t), \ \forall  x_i(0),~v_i(0)\in\mathbb{R}^p. \label{elia-algo-q}
\end{align}
\end{subequations}
Similar results as shown in Theorem \ref{thm2} could be obtained. We omit the details due to space limitations.

Different from the requirement that $v_i(0)={\bf0}_p$ in the algorithm \eqref{kia-algo}, $v_i(0)$ can be arbitrarily chosen in the algorithm  \eqref{elia-algo}. In other words, the algorithm  \eqref{elia-algo} is robust to the initial condition $v_i(0)$. However, the algorithm  \eqref{elia-algo} requires additional communication of $v_j$ in \eqref{elia-algo-q}, compared to the algorithm \eqref{kia-algo}.
\end{remark}

\begin{remark}
In \cite{Jie-ZGS,yu2016gradient,kia2015distributed,zhang2017distributed,li2018distributed,Yi2018distributed,liang2019exponential},  the exponential convergence for continuous-time distributed algorithms was also established. However, in \cite{Jie-ZGS,yu2016gradient,kia2015distributed,zhang2017distributed}, it was assumed that each local cost function is strongly convex; in \cite{li2018distributed}, it was assumed that the global cost function is strongly convex; in \cite{Yi2018distributed}, it was assumed that the global cost function is restricted strongly convex and the optimal set is a singleton; and in \cite{liang2019exponential}, it was assumed that each local cost function is convex and the primal-dual gradient map is metrically subregular.
In contrast, the exponential convergence result established in Theorem~\ref{thm2} only requires the global cost function satisfies the restricted secant inequality condition, but the convexity assumption on cost functions and the singleton assumption on the optimal set are not required. The potential drawback of Theorem~\ref{thm2} (as well as Theorem~\ref{thm4} in the next section) is the requirement that $\alpha>\frac{8nM^2+\nu M}{\nu\rho_2(L)}$, which uses the global information. From Remark~\ref{cdc18so:remarkalpha}, we know that this requirement can be relaxed by $\alpha>0$ if each local cost function is convex.
\end{remark}

\section{Discrete-time Distributed algorithm}\label{sec-main-dc}

Consider a discretization of the continuous-time algorithm \eqref{kia-algo-compact} by the Euler's approximation as
\begin{subequations}\label{kia-algo-dc-compact}
\begin{align}
\bm{x}(k+1)=&\bm{x}(k)-h(\alpha\bsL\bm{x}(k)+\beta\bm{v}(k)+\nabla \tilde{f}(\bm{x}(k))),\\
\bm{v}(k+1)=&\bm{v}(k)+h\beta\bsL\bm{x}(k),~\forall \bsx(0)\in\mathbb{R}^{np},~\bsv(0)={\bf 0}_{np},
\end{align}
\end{subequations}
where $h>0$ is a fixed stepsize. It is straightforward to check that the algorithm \eqref{kia-algo-dc-compact} is equivalent to the algorithm EXTRA proposed in \cite{shi2015extra} with mixing matrices $W={\bf I}_{np}-h\alpha\bsL$ and $\tilde{W}={\bf I}_{np}-h\alpha\bsL+h^2\beta^2\bsL$.
The distributed form of \eqref{kia-algo-dc-compact} is
\begin{subequations}\label{kia-algo-dc}
\begin{align}
x_i(k+1) =& x_i(k)-h(\alpha\sum_{j=1}^nL_{ij}x_j(k)+\beta v_i(k)\notag\\
& +\nabla f_i(x_i(k))), \label{kia-algo-dc-x}\\
v_i(k+1) =&v_i(k)+ h\beta\sum_{j=1}^n L_{ij}x_j(k),\notag\\
 &\forall x_i(0)\in\mathbb{R}^p, \ v_i(0)={\bf 0}_p.  \label{kia-algo-dc-q}
\end{align}
\end{subequations}

We have the following result for the discrete-time distributed primal-dual gradient decent algorithm \eqref{kia-algo-dc}.
\begin{theorem}\label{thm4}
Each agent $i\in [n]$ runs the distributed algorithm \eqref{kia-algo-dc}. If Assumptions~\ref{cdc18so:ass:filc}--\ref{cdc18so:ass:singleton} hold, $\alpha>\frac{2nL_{f}^2+\nu L_{f}}{\nu\rho_2(L)}$, $\beta>0$, and $0<h<\frac{2\epsilon_2\epsilon_4}{\eta\epsilon_3\epsilon_5}$, where $\eta=\sqrt{2}\max\{\frac{2\epsilon_1}{\rho_2(L)}+\alpha+1,4\epsilon_1+1\}>0$, $\epsilon_4=\epsilon_1\min\{\frac{1}{\rho(L)},\frac{1}{2}\}$, and $\epsilon_5=\max\{\beta^2\rho^2(L)+3\alpha^2\rho^2(L)+3L_{f}^2,3\beta^2\}$, then $\bsx(k)$ linearly converges to $\bsX^*$  with a rate no less than $1-\frac{h(2\epsilon_2\epsilon_4-h\eta\epsilon_3\epsilon_5)}{4\epsilon_3\epsilon_4}$.
\end{theorem}
\begin{proof}
The proof is given in Appendix~\ref{proof-thm4}.
\end{proof}

\begin{remark}
In \cite{xin2019frost,jakovetic2015linear,nedic2017achieving,nedic2017geometrically,qu2018harnessing,qu2017accelerated,xi2018add,xu2018convergence,xin2018linear,pu2018push,jakovetic2019unification,Damiano-TAC2016,saadatniaki2018optimization,zeng2017extrapush,xi2017dextra,shi2015extra},  the linear convergence for distributed discrete-time algorithms was also established. However, in \cite{xin2019frost,jakovetic2015linear,nedic2017achieving,nedic2017geometrically,qu2018harnessing,qu2017accelerated,xi2018add,xu2018convergence,xin2018linear,pu2018push,jakovetic2019unification}, it was assumed that each local cost function is strongly convex; in \cite{Damiano-TAC2016,saadatniaki2018optimization}, it was assumed that the global cost function is strongly convex; in \cite{zeng2017extrapush,xi2017dextra}, it was assumed that each local cost function is restricted strongly convex and the optimal set $X^*$ is a singleton; and in \cite{shi2015extra}, it was assumed that the global cost function is restricted strongly convex and $X^*$ is a singleton. In contrast, the linear convergence result established in Theorem~\ref{thm4} only requires the global cost function satisfies the restricted secant inequality condition, but the convexity assumption on cost functions and the singleton assumption on the optimal set are not required. One potential drawback of Theorem~\ref{thm4} is that the requirement on the stepsize $h$ is too conservative.
\end{remark}

\section{Simulations}\label{sec-simulation}
In this section, we verify the theoretical result through a numerical example.
Consider the distributed optimization problem \eqref{eqn:xopt} with
\begin{align*}
f_i(x)=\begin{cases}
         b_{i,1}(x+1)^2, & x\le-1,\\
         b_{i,2}x^4, & -1< x\le0, \\
         1-\sqrt{1-x^2}+b_{i,3}x^2, & 0\le x<\frac{\sqrt{2}}{2}, \\
         g_{i,1}(x), & \frac{\sqrt{2}}{2}\le x<1,\\
         g_{i,2}(x), & x\ge1,
       \end{cases}
\end{align*}
where $g_{i,1}(x)=\sqrt{1-(x-\sqrt{2})^2}-\sqrt{2}+1+b_{i,3}x^2$, $g_{i,2}(x)=\frac{1}{2}(x-1+\sqrt{\frac{\sqrt{2}-1}{2}})^2+\sqrt{2\sqrt{2}-2}+\frac{5-5\sqrt{2}}{4}+b_{i,3}x^2$, and $b_{i,j},~j=1,2,3$ are constants that are randomly generated and satisfy the condition that $\sum_{i=1}^nb_{i,1}>0$ and $\sum_{i=1}^nb_{i,2}=\sum_{i=1}^nb_{i,3}=0$. These $f_i(x),~i\in[n]$ are modifications of  Example~\ref{nonconvex:example}.
Clearly, $f_i$ is non-convex but differentiable with
\begin{align*}
\nabla f_i(x)=\begin{cases}
         2b_{i,1}(x+1), & x\le-1,\\
         4b_{i,2}x^3, & -1< x\le0, \\
         \frac{x}{\sqrt{1-x^2}}+2b_{i,3}x, & 0\le x<\frac{\sqrt{2}}{2}, \\
         \frac{\sqrt{2}-x}{\sqrt{1-(x-\sqrt{2})^2}}+2b_{i,3}x, & \frac{\sqrt{2}}{2}\le x<1,\\
         (2b_{i,3}+1)x-1+\sqrt{\frac{\sqrt{2}-1}{2}}, & x\ge1.
       \end{cases}
\end{align*}
It is easy to see that the global objective $f(x)=\sum_{i=1}^nf_i(x)$ satisfies the restricted secant inequality condition with constant $\nu=\min\{\sqrt{\frac{\sqrt{2}-1}{2}},2\sum_{i=1}^nb_{i,1}\}$. Moreover, the optimal set is $[-1,0]$. The communication graph between agents is modeled as a ring graph with $n=10$ agents, see Fig.~\ref{nonconvex:fig:network}.

\begin{figure}[!ht]
\hspace{10mm}\centering{
{
\begin{tikzpicture}[-,node distance=1.4cm,
  thick,main node/.style={circle,fill=yellow!20,draw,font=\sffamily\normalsize\bfseries}]
  \node[main node] (1) {1};
  \node[main node] (2) [above right of=1] {2};
  \node[main node] (3) [right of=2] {3};
  \node[main node] (4) [right of=3] {4};
  \node[main node] (5) [below right of=4] {5};
  \node[main node] (6) [below of=5] {6};
  \node[main node] (7) [below left of=6] {7};
  \node[main node] (8) [left of=7] {8};
  \node[main node] (9) [left of=8] {9};
  \node[main node] (10) [above left of=9] {10};
\draw (1) -- (2)
	  (2) -- (3)
	  (3) -- (4)
	  (4) -- (5)
	  (5) -- (6)
	  (6) -- (7)
	  (7) -- (8)
	  (8) -- (9)
	  (9) -- (10)
	  (10) -- (1);
	
\end{tikzpicture}\label{fig:network-full}}}$\qquad\qquad$
\vspace{3mm}
\caption{The communication graph.}
\label{nonconvex:fig:network}
\end{figure}
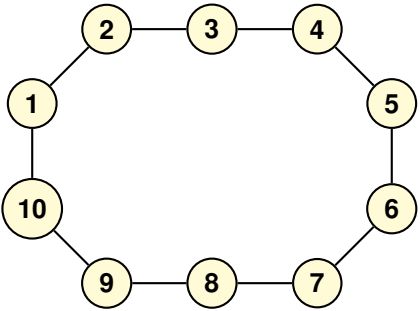

We run the discrete-time distributed primal-dual gradient decent algorithm \eqref{kia-algo-dc} with $\alpha=\beta=10$ and $h=0.02$. The initial value $x_i(0)$ is randomly generated. The trajectories of the primal and dual variables of each agent are plotted in Fig.~\ref{nonconvex:fig:x-traj} and Fig.~\ref{nonconvex:fig:v-traj}, respectively. We see that each primal variable converges to zero which is one global minimizer and correspondingly each dual variable also converges to zero. Evolutions of residual $\|\bsx(k)-\calP_{\bsX^*}(\bsx(k))\|/\|\bsx(0)-\calP_{\bsX^*}(\bsx(0))\|$ are shown in Fig.~\ref{nonconvex:fig:performance}. The results illustrate linear convergence, which are consistent with the theoretical results of Theorem~\ref{thm4}.

\begin{figure}[!ht]
\centering
  \includegraphics[width=0.47\textwidth]{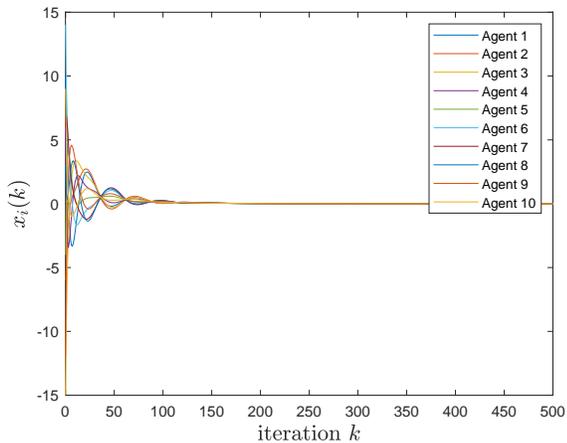}
  \caption{Evolutions of local primal variables.}
  \label{nonconvex:fig:x-traj}
\end{figure}

\begin{figure}[!ht]
\centering
  \includegraphics[width=0.47\textwidth]{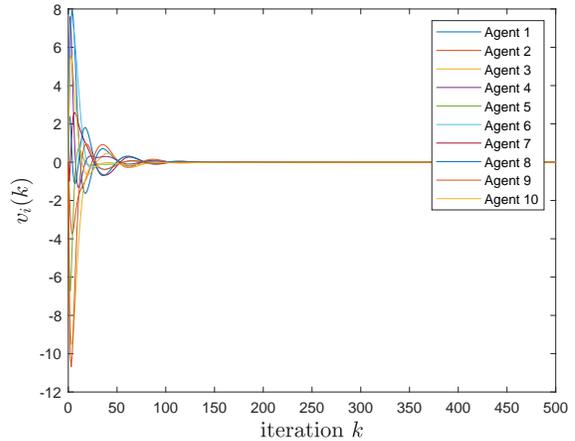}
  \caption{Evolutions of local dual variables.}
  \label{nonconvex:fig:v-traj}
\end{figure}

\begin{figure}[!ht]
\centering
  \includegraphics[width=0.47\textwidth]{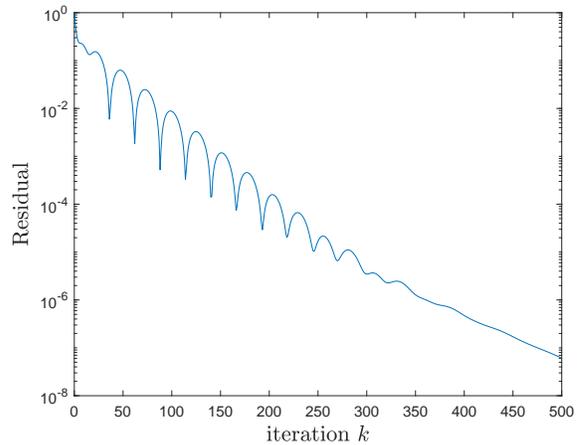}
  \caption{Evolutions of residual.}
  \label{nonconvex:fig:performance}
\end{figure}

\section{Conclusions}\label{sec-conclusion}
In this paper, we derived the exponential convergence rate of the continuous-time distributed primal-dual algorithm for solving distributed smooth optimization when the global cost function  satisfies the
restricted secant inequality condition. This condition relaxes the standard strong convexity condition.  We also showed that the discrete-time counterpart of the continuous-time algorithm establishes linear convergence rate under the same condition. Interesting open questions for future work include proving the linear convergence rate for larger stepsize, considering asynchronous and dynamic network setting, studying constraints, and relaxing the restricted secant inequality condition by the Polyak-{\L}ojasiewicz condition.

\bibliographystyle{IEEEtran}
\bibliography{refextra}

\appendix

\subsection{Useful Lemmas}\label{app-lemmas}

\begin{lemma}\label{lemma-Xinlei} (Lemmas~1 and 2 in \cite{Yi2018distributed})
Let $L$ be the Laplacian matrix of the connected graph $\mathcal{G}$ and $K_n=\bsI_n-\frac{1}{n}{\bf 1}_n{\bf 1}^{\top}_n$.
Then $L$ and $K_n$ are positive semi-definite, $\nullrank(L)=\nullrank(K_n)=\{{\bf 1}_n\}$, $L\le\rho(L)I_n$,
\begin{align}\label{KL-L-eq}
&K_nL=LK_n=L,~\rho(K_n)=1,\notag\\
&\text{and}~0\le\rho_2(L)K_n\le L\le\rho(L)K_n.
\end{align}
Moreover, there exists an orthogonal matrix $Q=[r \ R]\in \mathbb{R}^{n \times n}$ with $r=\frac{1}{\sqrt{n}}\mathbf{1}_n$ and $R \in \mathbb{R}^{n\times (n-1)}$ such that
\begin{equation}
L=[r \ R]
\begin{bmatrix}
0 & 0\\
0 & \Lambda_1
\end{bmatrix}
\begin{bmatrix}
r^{\top}\\
R^{\top}
\end{bmatrix},\label{cdc18so:lemma:LReq}
\end{equation}
\begin{equation}\label{lemma-eq}
R\Lambda_1^{-1}R^{\top}L=LR\Lambda_1^{-1}R^{\top}=K_n,
\end{equation}
\begin{equation}\label{lemma-eq2}
\frac{1}{\rho(L)}K_n\leq R\Lambda_1^{-1}R^{\top}\le\frac{1}{\rho_2(L)}K_n,
\end{equation}
where $\Lambda_1=\diag\left([\lambda_2,\dots,\lambda_n]\right)$ with $0<\lambda_2\leq\dots\leq\lambda_n$ are the eigenvalues of the Laplacian matrix $L$, and $\sqrt{\Lambda_1}=\diag\left([\sqrt{\lambda_2},\dots,\sqrt{\lambda_n}]\right)$.
\end{lemma}

\begin{lemma}\label{cdc18so:lemma:projection} (Theorem~1.5.5 in \cite{facchinei2007finite})
Let $\mathbb{S}$ be a nonempty closed convex subset of $\mathbb{R}^{p}$. Then, (i) for each $x\in\mathbb{R}^p$, the projection $\calP_{\mathbb{S}}(x)$  exists and is unique; (ii) $\calP_{\mathbb{S}}(x)$ is nonexpansive, i.e., $\|\calP_{\mathbb{S}}(a)-\calP_{\mathbb{S}}(b)\|\le\|a-b\|,~\forall a,b\in\mathbb{R}^p$; (iii) the squared distance function $g(x)=\|x-\calP_{\mathbb{S}}(x)\|^2$ is continuously differentiable and $\nabla g(x)=2(x-\calP_{\mathbb{S}}(x))$.
\end{lemma}

\begin{lemma}\label{cdc18so:lemma:lipschitz} (Lemma~1.2.3 in \cite{nesterov2018lectures})
If the function $g:\mathbb{R}^{p}\rightarrow \mathbb{R}$ is differentiable and smooth with constant $\eta>0$, then
\begin{align*}
|g(y)-g(x)-(y-x)^\top\nabla g(x)|\le\frac{\eta}{2}\|y-x\|^2,~\forall x,y\in\mathbb{R}^{p}.
\end{align*}
\end{lemma}

\subsection{Proof of Lemma~\ref{cdc18so:prop}}\label{cdc18so:propproof}
The following proof is inspired the proof of Proposition~3.6 in \cite{shi2015extra} and the key challenge is that here $\tilde{f}$ may be non-convex  and $X^*$ may not be a singleton.

For any $\bsx\in \mathbb{R}^{np}$, orthogonally decompose it as
\begin{align}\label{cdc18so:lemma:projectionconeq}
\bsx=\bsa+\bsb,
\end{align}
so that $\bsa\in\mathbb{H}$ and $\bsb\in\mathbb{H}^\bot$, where $\mathbb{H}=\{{\bf 1}_n \otimes x:~x\in\mathbb{R}^p\}$ is a linear subspace of $\mathbb{R}^{np}$. It is straightforward to check that such a decomposition is unique and $\bsa={\bf 1}_n \otimes a$ with $a=\frac{1}{n}({\bf 1}_n^\top\otimes \bsI_p)\bsx$. Moreover, it holds that $\calP_{\bsX^*}(\bsx)=\calP_{\bsX^*}(\bsa)$ since
\begin{align*}
&\min_{\bsc\in\bsX^*}\|\bsc-\bsx\|^2=\min_{\bsc\in\bsX^*}\|\bsc-\bsa-\bsb\|^2\\
&=\|\bsb\|^2+\min_{\bsc\in\bsX^*}\|\bsc-\bsa\|^2,
\end{align*}
where the last equality holds due to that $\bsc\in\bsX^*\subseteq\mathbb{H}$, $\bsa\in\mathbb{H}$, and $\bsb\in\mathbb{H}^\bot$.  For convenience, in the following let $\bsx^*={\bf 1}_n \otimes x^*=\calP_{\bsX^*}(\bsx)$. Then, $\calP_{X^*}(a)=x^*$ and
\begin{align*}
\|\bsx-\bsx^*\|^2=\|\bsa-\bsx^*\|^2+\|\bsb\|^2.
\end{align*}
From Assumption~\ref{cdc18so:ass:fil}, we have
\begin{align}\label{cdc18so:prop:equf1}
&(\nabla \tilde{f}(\bsa)-\nabla \tilde{f}(\bsx^*))^\top(\bsa-\bsx^*)\nonumber\\
&=\sum_{i=1}^{n}(\nabla f_i(a)-\nabla f_i(x^*))^\top(a-x^*)\nonumber\\
&=(\nabla f(a)-\nabla f(x^*))^\top(a-x^*)\nonumber\\
&\ge \nu\|a-x^*\|^2=\frac{\nu}{n}\|\bsa-\bsx^*\|^2.
\end{align}
From Assumption~\ref{cdc18so:ass:fiu}, we have
\begin{align}
(\nabla \tilde{f}(\bsx)-\nabla \tilde{f}(\bsa))^\top(\bsx-\bsa)&\ge-L_{f}\|\bsx-\bsa\|^2\notag\\
&=-L_{f}\|\bsb\|^2,\label{cdc18so:prop:equf2}\\
(\nabla \tilde{f}(\bsa)-\nabla \tilde{f}(\bsx^*))^\top(\bsx-\bsa)
&\ge-L_{f}\|\bsa-\bsx^*\|\|\bsb\|,\label{cdc18so:prop:equf3}\\
(\nabla \tilde{f}(\bsx)-\nabla \tilde{f}(\bsa))^\top(\bsa-\bsx^*)
&\ge-L_{f}\|\bsb\|\|\bsa-\bsx^*\|.\label{cdc18so:prop:equf4}
\end{align}
Hence, from \eqref{cdc18so:prop:equf1}--\eqref{cdc18so:prop:equf4}, we have
\begin{align}\label{cdc18so:prop:equf5}
&(\nabla \tilde{f}(\bsx)-\nabla \tilde{f}(\bsx^*))^\top(\bsx-\bsx^*)\nonumber\\
&=(\nabla \tilde{f}(\bsa)-\nabla \tilde{f}(\bsx^*))^\top(\bsa-\bsx^*)\nonumber\\
&~~~+(\nabla \tilde{f}(\bsx)-\nabla \tilde{f}(\bsa))^\top(\bsx-\bsa)\nonumber\\
&~~~+(\nabla \tilde{f}(\bsa)-\nabla \tilde{f}(\bsx^*))^\top(\bsx-\bsa)\nonumber\\
&~~~+(\nabla \tilde{f}(\bsx)-\nabla \tilde{f}(\bsa))^\top(\bsa-\bsx^*)\nonumber\\
&\ge \frac{\nu}{n}\|\bsa-\bsx^*\|^2-2L_{f}\|\bsa-\bsx^*\|\|\bsb\|-L_{f}\|\bsb\|^2\nonumber\\
&\ge \frac{\nu}{2n}\|\bsa-\bsx^*\|^2-(\frac{2nL_{f}^2}{\nu}+L_{f})\|\bsb\|^2.
\end{align}
From \eqref{cdc18so:lemma:LReq} and $\bsb\in\mathbb{H}^\bot$, we have
\begin{align}\label{cdc18so:prop:equf6}
&\bsx^\top\bsL\bsx=\bsb^\top\bsL\bsb\nonumber\\
&=\bsb^\top(R\Lambda_1 R^\top\otimes I_p)\bsb
\ge\rho_2(L)\bsb^\top(RR^\top\otimes I_p)\bsb\nonumber\\
&=\rho_2(L)\bsb^\top((\bsI_n-\frac{1}{n}{\bf 1}_n{\bf 1}^{\top}_n)\otimes I_p)\bsb=\rho_2(L)\|\bsb\|^2.
\end{align}
Hence, \eqref{cdc18so:prop:equf5} and \eqref{cdc18so:prop:equf6} yield \eqref{rsc-eq}.

\subsection{Proof of Theorem~\ref{thm2}}\label{proof-thm2}

Consider on the following functions
\begin{align}
V_1(\bm{x},\bm{v})=&\frac{1}{2}\|\bm{x}- \bm{x}^0 \|^2+\frac{1}{2}\|\bm{v} - \bm{v}^0\|^2_{R\Lambda^{-1}_1R^{\top}\otimes I_p},\label{func-V-q}\\
V_2(\bm{x},\bm{v})=&2\epsilon_1V_1(\bm{x},\bm{v})+\frac{\alpha}{2\beta}\|\bm{v}-\bm{v}^0\|^2,\label{lyap-q-eq3}\\
V_3(\bm{x},\bm{v})=&\bm{x}^\top(K_n\otimes \bsI_p)(\bm{v}-\bm{v}^0),\label{lyap-eq2}
\end{align}
where $R$ and $\Lambda$ are defined in Lemma~\ref{lemma-Xinlei},
$\bsx^0={\bf 1}_n\otimes x^0=\calP_{\bsX^*}(\bsx)$, and $\bm{v}^0=-\frac{1}{\beta}\nabla f (\bm{x}^0)$.

The derivative of $V_1(\bm{x},\bm{v})$ along the trajectories of \eqref{kia-algo-compact} satisfies
\begin{align}
&\dot{V}_1\nonumber\\
&=(\bm{x}-\bm{x}^0)^{\top}(-\alpha\bsL\bm{x}-\beta\bm{v}-\nabla \tilde{f}(\bm{x}))\notag \\
&~~~+\beta(\bm{v}-\bm{v}^0)^{\top}(K_n\otimes \bsI_p)(\bm{x}-\bm{x}^0)\nonumber\\
&=(\bm{x}-\bm{x}^0)^{\top}(-\alpha\bsL\bm{x}-\beta(\bm{v}-\bm{v}^0)\notag \\
&~~~-(\nabla \tilde{f}(\bm{x})-\nabla \tilde{f}(\bm{x}^0))+\beta(\bm{v}-\bm{v}^0)^{\top}(\bm{x}-\bm{x}^0)\nonumber\\
&=-\alpha\bm{x}^{\top}\bsL\bm{x}-(\bm{x}-\bm{x}^0)^{\top}(\nabla \tilde{f}(\bm{x})-\nabla \tilde{f}(\bm{x}^0))\nonumber\\
&\le-\nu_1\|\bm{x}-\bm{x}^0\|^2,\label{vdot-eq1-q}
\end{align}
where the first equality follows from Lemma~\ref{cdc18so:lemma:projection}, \eqref{lemma-eq}, and $(K_n\otimes \bsI_p)\bm{x}^0={\bf 0}_{np}$; the second equality follows from $\beta\bm{v}^0=-\nabla \tilde{f}(\bm{x}^0)$, $(\bm{x}^0)^{\top} \bsL ={\bf0}_{np}$, and $(K_n\otimes \bsI_p ) (\bm{v}-\bm{v}^0)=\bm{v}-\bm{v}^0$ since the facts that $\sum_{i=1}^{n} v_i(t)$ remains unchanged with respect to $t$ and the initial states satisfy $\sum_{i=1}^{n}v_i(0)={\bf0}_p$; and the inequality follows from \eqref{rsc-eq}.

Similarly, we know that the derivatives of $V_2$ and $V_3$ along the trajectories of \eqref{kia-algo-compact} satisfy
\begin{align}
&\dot{V}_2\nonumber\\
&\le-2\epsilon_1\nu_1 \|\bm{x}-\bm{x}^0\|^2+\alpha(\bm{v}-\bm{v}^0)^\top\bsL\bm{x}(t),\label{lyap-q-eq3dot}\\
&\dot{V}_3\nonumber\\
&=(\bm{v}-\bm{v}^0)^\top(K_n\otimes \bsI_p)(-\alpha\bsL\bm{x}-\beta\bm{v}-\nabla \tilde{f}(\bm{x}))\notag \\
&~~~+\beta(\bm{x}-\bm{x}^0)^\top\bsL\bm{x}\notag \\
&=(\bm{v}-\bm{v}^0)^\top(K_n\otimes \bsI_p)(-\alpha\bsL\bm{x}-\beta(\bm{v}-\bm{v}^0)\notag \\
&~~~-(\nabla \tilde{f}(\bm{x})-\nabla \tilde{f}(\bm{x}^0)))+\beta(\bm{x}-\bm{x}^0)^\top\bsL(\bm{x}-\bm{x}^0)\notag \\
&\le-\beta\|\bm{v}-\bm{v}^0\|^2-\alpha(\bm{v}-\bm{v}^0)^\top\bsL\bm{x}(t)+\rho(L)\beta\|\bm{x}-\bm{x}^0\|^2\notag\\
&~~~+\frac{\beta}{2}\|\bm{v}-\bm{v}^0\|^2\notag+\frac{1}{2\beta}\|\nabla\tilde{f}(\bm{x})-\nabla \tilde{f}(\bar{x}^0)\|^2\notag\\
&\le-\frac{\beta}{2}\|\bm{v}-\bm{v}^0\|^2-\alpha(\bm{v}-\bm{v}^0)^\top\bsL\bm{x}(t)\notag\\
&~~~+(\rho(L)\beta+\frac{L_{f}^2}{2\beta})\|\bm{x}-\bm{x}^0\|^2.\label{lyap-q-eq2dot}
\end{align}

Consider the following Lyapunov candidate
\begin{equation}\label{lyap-eq-augv2}
V(\bm{x},\bm{v})=V_2(\bm{x},\bm{v})+V_3(\bm{x},\bm{v}).
\end{equation}
From  \eqref{lyap-q-eq2dot}--\eqref{lyap-q-eq3dot}, we know that the derivative of $V$ along the trajectories of \eqref{kia-algo-compact} satisfies
\begin{align}
\dot{V} \leq &-\frac{\beta}{2}\|\bm{v}-\bm{v}^0\|^2-\epsilon_1\nu_1 \|\bm{x}-\bm{x}^0\|^2\notag\\
\le&\epsilon_2(\|\bm{v}-\bm{v}^0\|^2+ \|\bm{x}-\bm{x}^0\|^2).\label{lyap-q-v2v3}
\end{align}
From the Young's inequality, $(K_n\otimes \bsI_p )(\bm{v}-\bm{v}^0)=\bm{v}-\bm{v}^0$, $(\bm{x}^0)^{\top} (K_n\otimes \bsI_p)={\bf 0}_p$, and \eqref{lemma-eq2}, we have that
\begin{align}
&\epsilon_4(\|\bm{v}-\bm{v}^0\|^2+ \|\bm{x}-\bm{x}^0\|^2)\label{lyap-eq-v2low}\\
&\le V \le\epsilon_3(\|\bm{v}-\bm{v}^0\|^2+ \|\bm{x}-\bm{x}^0\|^2),\label{lyap-eq-v2up}
\end{align}
where $\epsilon_4=\epsilon_1\min\{\frac{1}{\rho(L)},\frac{1}{2}\}$.
Then, \eqref{lyap-q-v2v3} and \eqref{lyap-eq-v2up} yield
\begin{align}\label{lyap-q-v2v3v}
\dot{V}\leq -\frac{\epsilon_2}{\epsilon_3}V.
\end{align}
Thus, $V(t)\le V(0)e^{-\frac{\epsilon_2}{\epsilon_3}t}$. Noting that $\|\bm{x}-\bm{x}^0\|^2\le \frac{1}{\epsilon_4}V$, we know that $\bsx(t)$ exponentially converges to $\bsX^*$ with a rate no less than $\frac{\epsilon_2}{2\epsilon_3}>0$.

\subsection{Proof of Theorem~\ref{thm4}}\label{proof-thm4}

Denote
\begin{align*}
\bsF(\bsz)=\begin{bmatrix}
\alpha\bsL\bm{x}+\beta\bm{v}+\nabla \tilde{f}(\bm{x})\\
-\beta\bsL\bm{x}
\end{bmatrix},
\end{align*}
then we can rewrite \eqref{kia-algo-dc-compact} as
\begin{align}\label{kia-algo-dc-compactz}
\bsz(k+1)=\bsz(k)-h\bsF(\bsz(k)).
\end{align}

$V$ defined in \eqref{lyap-eq-augv2} is differentiable and its gradient is
\begin{align*}
\nabla V(\bsz)=\begin{bmatrix}
[\nabla V(\bsz)]_1\\
[\nabla V(\bsz)]_2
\end{bmatrix},
\end{align*}
where $[\nabla V(\bsz)]_1=2\epsilon_1(\bm{x}-\bsx^0)+(K_n\otimes \bsI_p)(\bsv-\bsv^0)$ and $[\nabla V(\bsz)]_2=(K_n\otimes \bsI_p)\bm{x}+(2\epsilon_1(R\Lambda^{-1}_1R^{\top}\otimes \bsI_p)+\frac{\alpha}{\beta}\otimes\bsI_{np})( \bm{v} - \bm{v}^0 )$.
Noting that the projection is nonexpansive as shown in Lemma~\ref{cdc18so:lemma:projection} and $\bsv^0$ is a constant vector, we know that $\nabla V(\bsz)$ is Lipschitz continuous with constant $\eta$. Then, from Lemma~\ref{cdc18so:lemma:lipschitz}, we have that
\begin{align}\label{cdc18so:v2ds}
&V(\bsz(k+1))-V(\bsz(k))\notag\\
&\le(\bsz(k+1)-\bsz(k))^\top \nabla V(\bsz(k))+\frac{\eta}{2}\|\bsz(k+1)-\bsz(k)\|^2\notag\\
&=-h\bsF^\top(\bsz(k))\nabla V(\bsz(k))+\frac{h^2\eta}{2}\|\bsF(\bsz(k))\|^2.
\end{align}
From \eqref{lyap-q-v2v3v}, we know that
\begin{align}\label{cdc18s0:v2ds1}
-\bsF^\top(\bsz(k))\nabla V(\bsz(k))\leq -\frac{\epsilon_2}{\epsilon_3}V(\bsz(k)).
\end{align}
From $\bsL \bm{x}^0={\bf 0}_{np}$ and $\bm{v}^0=-\nabla \tilde{f}(\bm{x}^0)$, we know that
\begin{align*}
\bsF(\bsz)=\begin{bmatrix}
\alpha\bsL(\bm{x}-\bm{x}^0)+\beta(\bm{v}-\bm{v}^0)+\nabla \tilde{f}(\bm{x})-\nabla \tilde{f}(\bm{x}^0)\\
-\beta\bsL(\bm{x}-\bm{x}^0)
\end{bmatrix}.
\end{align*}
Hence, from  Lemma~\ref{lemma-Xinlei}, Assumption~\ref{cdc18so:ass:fiu}, and \eqref{lyap-eq-v2low}, we have that
\begin{align}\label{cdc18s0:v2ds2}
&\|\bsF^\top(\bsz)\|^2\notag\\
&\le \beta^2\rho^2(L)\|\bm{x}-\bm{x}^0\|^2+3\alpha^2\rho^2(L)\|\bm{x}-\bm{x}^0\|^2\notag\\
&~~~+3\beta^2\|\bm{v}-\bm{v}^0\|^2+3\|\nabla \tilde{f}(\bm{x})-\nabla \tilde{f}(\bm{x}^0)\|^2\notag\\
&\le (\beta^2\rho^2(L)+3\alpha^2\rho^2(L)+3L_{f}^2)\|\bm{x}-\bm{x}^0\|^2\notag\\
&~~~+3\beta^2\|\bm{v}-\bm{v}^0\|^2\notag\\
&\le \epsilon_5(\|\bm{v}-\bm{v}^0+ \|\bm{x}-\bm{x}^0\|^2)\notag\\
&\le \frac{\epsilon_5}{\epsilon_4}V_1(\bsz).
\end{align}

Then, from \eqref{cdc18so:v2ds}--\eqref{cdc18s0:v2ds2}, we have
\begin{align}\label{cdc18so:v2ds3}
&V(\bsz(k+1))\notag\\
&\le V(\bsz(k))-\frac{h\epsilon_2}{\epsilon_3}V(\bsz(k))+\frac{h^2\eta\epsilon_5}{2\epsilon_4}V(\bsz(k))\notag\\
&=(1-\frac{h(2\epsilon_2\epsilon_4-h\eta\epsilon_3\epsilon_5)}{2\epsilon_3\epsilon_4})V(\bsz(k))\notag\\
&\le(1-\frac{h(2\epsilon_2\epsilon_4-h\eta\epsilon_3\epsilon_5)}{2\epsilon_3\epsilon_4})^{k+1}V(\bsz(0)).
\end{align}
Thus,
\begin{align*}
&\|\bsx(k)-\calP_{\bsX^*}(\bsx(k))\|\le\sqrt{\frac{1}{\epsilon_4}V(\bsz(k))}\notag\\
&\le(1-\frac{h(2\epsilon_2\epsilon_4-h\eta\epsilon_3\epsilon_5)}{4\epsilon_3\epsilon_4})^{k}
\sqrt{\frac{1}{\epsilon_4}V(\bsz(0))}.
\end{align*}
In other words, $\bsx(k)$ linearly converges to $\bsX^*$ with a rate no less than $1-\frac{h(2\epsilon_2\epsilon_4-h\eta\epsilon_3\epsilon_5)}{4\epsilon_3\epsilon_4}$.

\end{document}